\documentclass[]{article}
\usepackage{lipsum}
\usepackage{amsfonts}
\usepackage{graphicx}
\usepackage{epstopdf}
\usepackage{algorithmic}
\usepackage{authblk}
\usepackage{enumitem}
\usepackage[colorlinks]{hyperref}
\hypersetup{citecolor = blue}

\usepackage{comment}
\usepackage{amsopn}
\usepackage{amssymb}
\usepackage{tikz-cd}

\usepackage[margin=1.5in]{geometry}

\usepackage{amsfonts,amsbsy,amscd,amsgen,dsfont,amsmath,amssymb,mathrsfs,amsthm}

\newcommand{\id}{\mathrm d}

\newcommand{\vc}{\mathbf}

\renewcommand{\tilde}{\widetilde}

\newcommand{\g}{\mathrm{g}}

\newcommand{\mdiff}[1]{\frac{\mathrm{D} #1}{\mathrm{D}t}}

\newcommand{\diff}[2]{\frac{\mathrm{d} #1}{\mathrm{d}#2}}

\newcommand{\eps}{\epsilon}
\newcommand{\N}{\mathbb{N}}
\newcommand{\dr}{R_\rho}
\newcommand{\vcw}{\tilde{\vc{w}}}
\newcommand{\epst}{{\tilde{\eps}}}

\DeclareMathAlphabet\mathbfcal{OMS}{cmsy}{b}{n}

\newcommand{\R}{\mathbb{R}}

\newcommand{\grad}{\nabla}

\theoremstyle{definition}
\newtheorem{theorem}{Theorem}
\newtheorem{remark}{Remark}
\newtheorem{definition}{Definition}

\begin{document}
	\title{Slow Manifold Reduction for Inertial Particles with Quadratic Drag}
	\author[]{Mathew Angel} 
	\author[]{Mohammad Farazmand\thanks{Corresponding author: farazmand@ncsu.edu}}
	\affil[]{Department of Mathematics, North Carolina State University, 2311 Stinson Drive, Raleigh, NC 27695, USA} 
	\renewcommand\Affilfont{\itshape\small}
	\date{}
	
	\maketitle	
	\maketitle
	
\begin{abstract}
We consider the dynamics of inertial particles in unsteady fluid flows.	At low Reynolds numbers, where the drag force is linear in the relative velocity, it is well-known that the dynamics admit an attracting, invariant, slow manifold which emerges as the perturbation of a normally hyperbolic critical manifold. However, at high Reynolds numbers, where the drag force is quadratic in the relative velocity, the critical manifold is no longer normally hyperbolic, and therefore its persistence has remained an open problem. Here, we resolve this issue by a particular application of the blowup method, which transforms the equations of motion to a generalized weighted cylindrical coordinate system, thereby desingularizing the dynamics on the critical manifold. We subsequently prove that the critical manifold persists under sufficiently small perturbations and derive the reduced equations of motion on the perturbed slow manifold to arbitrary accuracy. Our reduced equation differs from its linear-drag counterpart in its asymptotic expansion as well as its convergence rate. Using two examples, we demonstrate the validity of our slow manifold reduction. We also showcase an application of the reduced equations to the problem of source inversion in a turbulent dispersion model.
\end{abstract}

\section{Introduction}
Inertial particles refer to small rigid particles suspended in a fluid flow. Understanding their dynamics is crucial in many applications such as air pollution~\cite{Flagan1988}, sedimentation in water~\cite{Yue2009}, airborne disease transmission~\cite{Bourouiba2021}, and ocean drifters~\cite{Wichmann2019}, to name a few. 

Because of their inertia, the trajectories of these particles differ from the fluid trajectories. 
Consider a fluid with density $\rho_f$ filling the spatial domain $\Omega_x\subseteq\R^3$ and moving according to a known velocity field $\vc u:\R\times\Omega_x\rightarrow\R^3, (t,\vc x)\mapsto \vc u(t,\vc x)$. 
A small, rigid, spherical particle with radius $a$ and mass density $\rho_p$ is suspended in this fluid. We denote the position of the particle at time $t$ by $\vc x(t)\in\Omega_x$ and its velocity by $\vc v(t)=\dot{\vc x}(t)\in\mathbb R^3$. If the particle radius is small compared with the characteristic length scale of the flow, 
its motion is modeled by the force balance \cite{Kim1998,maxey1983equation},
\begin{equation}
\begin{aligned}
\rho_p\dot{\vc v}(t) = &\rho_f\mdiff{\vc u}(t,\vc x)\hspace{1cm}&&\text{(pressure gradient)}\\
&-(\rho_p-\rho_f)\g\vc{e}_\g&&\text{(buoyancy)}\\
&-\mathcal D[\vc v(t)-\vc u(t,\vc x)]&&\text{(drag)}\\
&-\frac{\rho_f}{2}\left[\dot{\vc v}(t)-\mdiff{\vc u}(t,\vc x)\right]&&\text{(added mass),}
\end{aligned}
\label{e1}
\end{equation}
where the material derivative $\mdiff{\ } :=\partial_t+\vc u\cdot \grad$ denotes the time derivative along the fluid trajectory. We denote the constant gravitational acceleration by $\g>0$ and choose the unit vector $\vc e_\g$ to point in the opposite direction of gravity. Each term on the right hand side of \eqref{e1} represents a distinct force acting on the particle, and the parenthetical labels identify the corresponding physical mechanism. Following common practice~\cite{IP_babiano,daitche2011memory}, model~\eqref{e1} neglects Fax\'en corrections and the Basset--Boussinesq history force; Fax\'en corrections can be incorporated in our slow manifold reduction in a straightforward fashion, whereas the history force introduces memory effects that cannot be treated in the same manner.

The drag law $\mathcal D[\vc v-\vc u]$ depends on the balance between viscous and inertial forces. For a particle moving with the characteristic speed $W$ relative to a fluid with kinematic viscosity $\nu$, the particle Reynolds number, $\mathrm{Re}_p = 2aW/\nu$, measures this balance. When $\mathrm{Re}_p$ is very small, viscous forces dominate, resulting in a linear drag which is proportional to the relative velocity of the particle. When $\mathrm{Re}_p$ is large, inertial forces dominate, and the drag force is quadratic in the relative velocity and proportional to a constant drag coefficient $C_d$. More specifically, the standard drag curve for a spherical particle \cite[Fig.~1.19]{schlichting1979boundary} dictates the following low and high particle Reynolds number regimes:
\begin{equation}
\mathcal D[\vc v-\vc u]=
\left\{
\begin{array}{lll}
\displaystyle \frac{9\nu\rho_f}{2a^2}(\vc v-\vc u),& \mathrm{Re}_p\lesssim1& \text{(linear drag)},\\
[1ex]\displaystyle 
\frac{3C_d\rho_f}{8a}\|\vc v-\vc u\|(\vc v-\vc u),& 10^3\lesssim\mathrm{Re}_p\lesssim2\times 10^5& \text{(quadratic drag)}.
\end{array}
\right.
\label{drag_law}
\end{equation}

For linear drag, geometric singular perturbation theory (GSPT) has been used to show that the particle dynamics admit an invariant slow manifold which attracts trajectories exponentially fast in time \cite{IP_haller08,mograbi2006,rubin1995_IP}. On the slow manifold, the particle velocity $\vc v(t)$ is dictated by its position and therefore can be written as a function of position $\vc x$ and time $t$. To leading order, the particle velocity field on the slow manifold is given by
\begin{equation}
\vc v(t) \simeq \vc u(t,\vc x)-\frac{2a^2}{9\nu}\cdot \frac{\rho_p-\rho_f}{\rho_f}\left(\mdiff{\vc u}+\g\vc e_\g\right).
\label{lin_slowman}
\end{equation}

For quadratic drag, however, GSPT does not apply as we discuss further in Section~\ref{sec:prelims}. Here, we use a geometric technique, known as the \emph{blowup construction}, to transform the problem into a form amenable to GSPT. Subsequently, we prove for the first time that system \eqref{e1} admits an attracting slow manifold under quadratic drag. To leading order, our slow manifold reduction is given by
\begin{equation}
\vc v(t) \simeq \vc u(t,\vc x)-\left(\frac{8a}{3C_d\rho_f|\rho_p-\rho_f|}\right)^\frac{1}{2}(\rho_p-\rho_f)\frac{\mdiff{\vc u}+\g\vc e_\g}{\|\mdiff{\vc u}+\g\vc e_\g\|^{1/2}}.
\end{equation}
We also derive the full slow manifold expansion, recursively generated higher order approximations.

In many applications, the particle Reynolds number is high enough that linear drag fails to capture the correct particle dynamics. Quadratic drag is therefore used in a variety of applications, including firebrand transport in wildfires \cite{Anthenien2006FirebrandTrajectories,Koo2012FirebrandTransport,Mendez2022}, sedimentation in turbulent flows \cite{Kok2009COMSALT,Zhao2022BedloadSaltation}, volcanic ash dispersal \cite{Beckett2015AshForecastSensitivity,Dioguardi2018IrregularParticleDrag}, and hail and snow settling \cite{Kumjian2020HailTrajectory,Tagliavini2021SnowDrag}. Hence, our results are relevant in a wide variety of applications. Compared to the full-order model~\eqref{e1}, our slow manifold reduction has several benefits: (i) It is lower dimensional and therefore computationally more efficient. (ii) It allows for source inversion through backward integration of the reduced dynamics which is impossible using the full-order model; see Section~\ref{sec:source_inv}.
(iii) It facilitates the identification of inertial particle clustering patterns, as it has already been done in the case of linear drag~\cite{IP_haller08,sapsis2010clustering}.

\subsection{Prior work}
In this section, we review two lines of prior work. The first concerns the development of the inertial particle model underlying equation \eqref{e1}. The second concerns the use of GSPT for the slow manifold reduction of inertial particle dynamics.

\subsubsection{Modeling aspects and drag laws}
The force balance in \eqref{e1}, with linear drag \eqref{drag_law}, arises from a series of models for the motion of small rigid spheres in viscous fluids. Stokes \cite{stokes1851} studied the viscous resistance on a sphere in creeping flow. Basset \cite{basset2}, Boussinesq \cite{boussinesq}, and Oseen \cite{oseen1927} developed force-balance models for the unsteady motion of a sphere in a viscous fluid, leading to the so-called BBO equation. Tchen \cite{Tchen} extended this framework to particles suspended in a turbulent flow by formulating the particle motion relative to the surrounding fluid velocity.

Later work refined the individual force terms. For example, Corrsin and Lumley \cite{corrsin1956} clarified the role of pressure gradients in \eqref{e1}. Maxey and Riley \cite{maxey1983equation} derived an equation for a small rigid sphere in a nonuniform, unsteady flow. Their formulation includes the pressure gradient, buoyancy, linear drag, added mass, Fax\'en corrections, and the Basset--Boussinesq history force. Auton et al. \cite{auton1988} later corrected the form of the added mass term, based on the eariler work by Taylor~\cite{Taylor1928}. The widely accepted form of the model is given in equation \eqref{e1} which we refer to as the Maxey--Riley equation, when linear drag is used.
The well-posedness of this model, including the Fax\'en corrections and the history force, was proved relatively recently~\cite{MR_EUR,langlois2015}.

The linear drag is only applicable in the so-called creeping flow, where the particle Reynolds number is very small. For larger particle Reynolds numbers, the drag law is often written using a drag coefficient $C_d$ that depends on the particle Reynolds number \cite{abraham1970}. This drag coefficient recovers linear drag in the low Reynolds number limit \cite{michaelides1997}. As the particle Reynolds number increases, the drag coefficient is well approximated by a constant, resulting in a drag force that is quadratic in the particle relative velocity. We refer to \eqref{e1}, with quadratic drag, as the \emph{quadratic Maxey--Riley equation}.

\subsubsection{Slow manifold reduction}
In the small inertia regime, inertial particle dynamics involve a fast adjustment of particle velocity and a slower evolution of particle position \cite{IP_haller08,mograbi2006}. After the fast transients, the particle velocity is given by a graph over position and time that involves the local fluid velocity and its derivatives. This graph defines a slow manifold, and the dynamics on it give a lower dimensional evolution equation for the inertial particle.

An early slow manifold reduction, in the linear drag regime, appears in Rubin, Jones and Maxey \cite{rubin1995_IP}. They study aerosol particles in a steady, spatially periodic cellular flow with gravity. Treating particle inertia as small, they prove the existence of a globally attracting slow manifold and use the reduced dynamics to describe the asymptotic settling velocity of the particles.

Mograbi and Bar-Ziv \cite{mograbi2006} consider the Maxey--Riley equation without the history term in the small inertia regime. They generalize the slow manifold reduction to steady flows with arbitrary spatial dependence. In particular, they identify a globally attracting invariant manifold and develop a recursive asymptotic scheme for computing it to an arbitrary order of accuracy.
Haller and Sapsis \cite{IP_haller08} generalized this reduction to unsteady flows, and used it to infer the regions where the particles accumulate~\cite{IP_haller08,sapsis2010clustering}.

This reduced inertial equation has since been used in numerous environmental and industrial transport problems~\cite{beron2015,Coletti2022,Vanneste2020,Wichmann2019}. For example, Sapsis and Haller \cite{sapsis2009inertial} apply it to inertial particle motion in a time-dependent hurricane model and use the reduced dynamics to identify inertial Lagrangian coherent structures. Tang et al. \cite{tang2009} use the same framework for atmospheric source inversion by projecting particle observations onto the slow manifold and integrating the reduced dynamics backward in time.

More recently, Aksamit et al. \cite{aksamit2026} derived a reduced-order model for snow particle transport from a Maxey--Riley-type equation with nonlinear drag. Their model uses the Reynolds number dependent Abraham drag law \cite{abraham1970}, which connects the linear drag regime at small Reynolds numbers to an approximately constant drag coefficient regime at large Reynolds numbers. However, their framework assumes that linear drag is dominant and therefore does not directly address the fully quadratic drag regime. As such, unlike the present work, their slow manifold reduction follows from a straightforward application of GSPT and does not require a blowup construction.

The slow manifold reductions described above either use linear drag directly or retain a Reynolds-dependent drag law whose dominant behavior is determined by linear drag \cite{abraham1970,aksamit2026,IP_haller08,mograbi2006}. The present work addresses  the high Reynolds number regime, where the drag coefficient $C_d$ is constant, and the drag law is genuinely quadratic. We show that, in this regime, GSPT is not immediately applicable because of loss of hyperbolicity. We further show that a nonlinear change of coordinates, known as the blowup construction, restores hyperbolicity and leads to a slow manifold reduction distinct from that of the linear drag setting.

\subsection{Outline}

The remainder of this paper is organized as follows. In Section~\ref{sec:prelims}, we formulate the model as a singular perturbation problem and show that, in the quadratic case, the slow manifold is non-hyperbolic. Section~\ref{sec:main_result} develops the blowup construction used to treat the quadratic drag system and derives the associated slow manifold reduction. Section~\ref{sec:numerics} presents our numerical and analytic examples that demonstrate the validity of the reduced model and apply it to a source inversion problem. Section~\ref{sec:conclusion} contains our concluding remarks.

\section{Preliminaries}
\label{sec:prelims}
This section sets up the singular perturbation formulation used throughout the paper. We first nondimensionalize the particle equation, write it equivalently in terms of the relative velocity, and formulate the resulting dynamics in the slow-fast setting. Applying GSPT to the linear and quadratic drag cases in parallel shows why the argument works for linear drag but fails for quadratic drag. This discussion motivates the blowup construction, as discussed in Section~\ref{sec:main_result}.

Given a characteristic velocity $U$, we define the characteristic length scale $L = U^2/\g$ and characteristic time scale $T_c=U/\g$. We then introduce dimensionless variables $\vc x\mapsto \vc x/L$, $\vc v\mapsto \vc v/U$, $\vc u\mapsto \vc u/U$, and $t\mapsto t/T_c$. We also introduce the relative velocity $\vc w(t) :=\vc v(t)-\vc u(\vc x(t),t)$ and write the system in terms of the dimensionless state variables $\vc x$ and  $\vc w$. The equations can be equivalently written in terms of the particle velocity $\vc v(t)$; however, the relative velocity formulation slightly simplifies the forthcoming analysis. We write the drag term as $\mu^{(\alpha)}\|\vc w\|^{\alpha-1}\vc w$, where $\alpha=1$ corresponds to linear drag and $\alpha=2$ corresponds to quadratic drag. These rearrangements of \eqref{e1} yield the equivalent formulation,
\begin{equation}
\begin{aligned}
\dot{\vc x} &= \vc w+\vc u,\\
\dot{\vc w} &=-\mu^{(\alpha)}\|\vc w\|^{\alpha-1}\vc w-\dr\left(\mdiff{\vc u}+\vc e_\g\right)-(\vc w\cdot\grad)\vc u,
\end{aligned}
\label{e2}
\end{equation}
with the dimensionless drag coefficients,
\begin{equation}
\mathcal \mu^{(\alpha)}=
\left\{\begin{array}{lll}\displaystyle 
\frac{3\nu U}{a^2 \g}(1-\dr)=\frac{3}{\mathrm{Re}_f}\left(\frac{L}{a}\right)^2(1-\dr),& \alpha=1& \text{(linear drag)},\\
[1ex]\displaystyle 
\frac{C_d}{4}\frac{U^2}{a\g}(1-\dr)=\frac{C_d}{4}\frac{L}{a}(1-\dr),& \alpha=2& \text{(quadratic drag)},
\end{array}
\right.
\end{equation}
where $\mathrm{Re}_f = UL/\nu$ is the fluid Reynolds number.

The density ratio $\dr = \frac{2(\rho_p-\rho_f)}{2\rho_p+\rho_f}$ satisfies $-2<\dr<1$ and measures the density contrast between the particle and the fluid. Neutrally buoyant particles have $\dr=0$. Particles lighter than the fluid are referred to as \emph{bubbles} and correspond to $\dr<0$, whereas particles heavier than the fluid are referred to as \emph{aerosols} and satisfy $\dr>0$. Since the particle radius $a$ is small compared to the fluid characteristic length scale $L$, the drag coefficients are large, $\mu^{(\alpha)}\gg 1$.
From this point on, for notational simplicity, we write $\mu$ instead of $\mu^{(\alpha)}$ whenever the value of $\alpha$ is clear from the context.
\begin{remark}
A few remarks are in order here. 
\begin{enumerate}
	\item In the linear drag regime, the inverse of the Stokes number, $\mathrm{St}:=\mathrm{Re}_f(a/L)^2$, appears as a prefactor in the drag coefficient. 
	\item In the literature, the nondimensional equations are often written in terms of the particle velocity $\vc v(t)$, instead of its relative velocity $\vc w(t)$. In that formulation, the density ratio is defined as $R = 2\rho_f/(\rho_f+2\rho_p)$ where neutrally buoyant particles correspond to $R=2/3$~\cite{IP_haller08}. Although the two formulations are equivalent, ours is more natural with $\dr=0$ for neurally buoyant particles.
\end{enumerate}
\end{remark}

\subsection{Slow-fast formulation}\label{sec:sf_form}
For $\mu\gg1$, we define the small parameter $\eps=1/\mu$ and rewrite equation \eqref{e2} as the singular perturbation problem,
\begin{equation}
\begin{aligned}
\dot{\vc x} &= \vc u+\vc w,\\
\eps\dot{\vc w} &=- \|\vc w\|^{\alpha-1}\vc w-\eps\dr\left(\mdiff{\vc u}+\vc e_\g\right)-\eps(\vc w\cdot\grad)\vc u.
\end{aligned}
\label{e4}
\end{equation}

We introduce the fast time $\tau=(t-t_0)/\eps$, so that differentiation with respect to $\tau$ satisfies 
$\id/\id \tau =\epsilon\, \id/\id t$, where $t_0>0$ is the initial time. We denote differentiation with respect to $\tau$ by a prime. As in the finite-time autonomous formulation used for inertial particle dynamics in unsteady flows \cite{IP_haller08}, we append $t$ as a state variable to obtain
\begin{subequations}
\begin{align}
t'&=\eps,\\
{\vc x}' &= \eps(\vc u+\vc w),\\
{\vc w}' &=-\|\vc w\|^{\alpha-1}\vc w-\eps\dr\left(\mdiff{\vc u}+\vc e_\g\right)-\eps(\vc w\cdot\grad)\vc u\label{e5b}.
\end{align}
\label{e5}
\end{subequations}

In \eqref{e5}, the relative velocity $\vc w$ evolves at an $\mathcal{O}(1)$ rate with respect to the fast time $\tau$. We therefore refer to $\vc w$ as the fast variable and to \eqref{e5b} as the fast equation. By contrast, the variables $t$ and $\vc x$ evolve at an $\mathcal{O}(\eps)$ rate with respect to the fast time. Therefore, we refer to them as the slow variables. 

\subsection{Critical manifold and normal hyperbolicity}

Setting $\eps=0$ in \eqref{e5} results in the so-called layer equations,
\begin{subequations}
	\label{e6}
\begin{align}
t'&=0,\\
{\vc x}' &=\vc 0,\\
{\vc w}' &=-\|\vc w\|^{\alpha-1}\vc w\label{e6b}.
\end{align}
\end{subequations}
The equilibria of \eqref{e6} occur when $\|\vc w\|^{\alpha-1}\vc w=0$. Thus, for any arbitrary $(t,\vc x)$, the points $(t,\vc x,\vc 0)$ are fixed points of the layer problem. To obtain a compact family of such fixed points, following~\cite{IP_haller08}, we assume that $\Omega_x$ is compact and define the finite time interval $\Omega_t = [t_0,t_0+T]$ for any finite $T>0$. Then the spatiotemporal domain $\Omega = \Omega_t\times\Omega_x$ is compact, and 
\begin{equation}\label{eq:crit_man}
\mathcal{M}_0 = \{(t,\vc x,\vc w):\quad\vc w=\vc 0,\quad (t,\vc x)\in \Omega\}
\end{equation}
is a compact invariant manifold consisting entirely of fixed points of \eqref{e6}. The set $\mathcal M_0$ is referred to as the \emph{critical manifold}. Figure~\ref{fig:slow_manifold_geometry}(a) illustrates this critical manifold as the $\vc w=\vc 0$ sheet in the extended phase space $(t,\vc x,\vc w)$. Fenichel's geometric singular perturbation theory \cite{fenichel1979} determines the conditions under which the critical manifold persists for $0<\eps\ll1$. 

\begin{theorem}[GSPT~\cite{fenichel1979} applied to \eqref{e5}] 
For $\epsilon=0$, suppose that $\mathcal{M}_0$ is an invariant, compact, normally hyperbolic, and attracting manifold of system \eqref{e5}. Then there exists $\eps_0>0$ such that for all $\eps\in[0,\eps_0)$,
\begin{enumerate}
\item There exists a locally invariant manifold $\mathcal{M}_{\eps}$ diffeomorphic to $\mathcal{M}_0$. 
\item $\mathcal{M}_{\eps}$ is as smooth as the right-hand side of \eqref{e5} and is $\mathcal{O}( \eps)$ close to $\mathcal{M}_0$ in Hausdorff distance.
\item $\mathcal{M}_{\eps}$ is normally hyperbolic and attracting. 
\end{enumerate}
\label{thm1}
\end{theorem}

Recall that local invariance means that trajectories can only enter or leave $\mathcal{M}_{\eps}$ through its boundaries. The manifold $\mathcal M_\eps$ is called a slow manifold because trajectories on it evolve on the slow time scale. Figure~\ref{fig:slow_manifold_geometry}(b) depicts the locally invariant slow manifold $\mathcal M_\eps$, which persists for $0<\eps\ll1$ and lies $\mathcal O(\eps)$ close to $\mathcal M_0$.

\begin{figure}[!t]
\centering
\includegraphics[width=\textwidth]{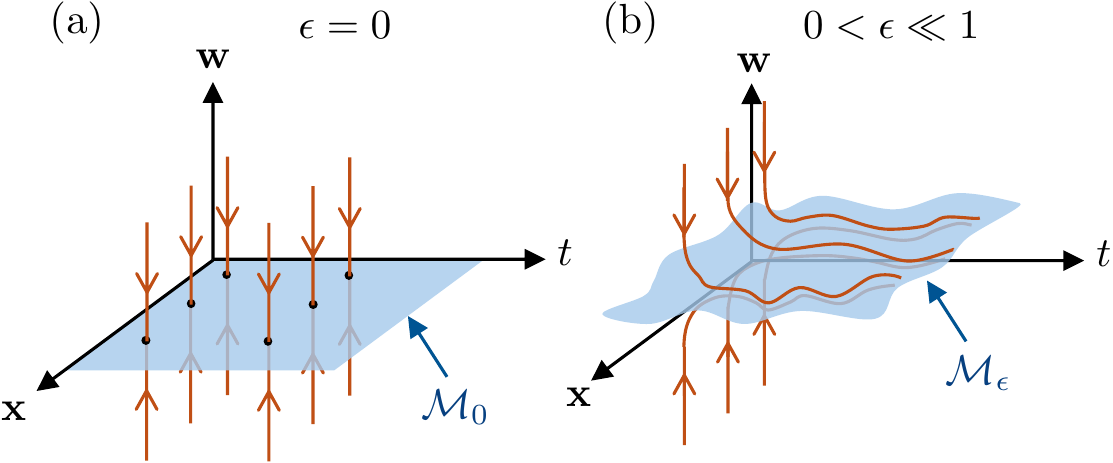}
\caption{Geometric depiction of GSPT for the slow-fast system
\eqref{e5}. (a) The critical manifold $\mathcal{M}_0$ comprising the set of equilibria of
the layer problem at $\eps=0$. (b) The nearby locally invariant slow manifold
$\mathcal{M}_\eps$ that exists for $0<\eps\ll1$ when $\mathcal{M}_0$ is normally hyperbolic. In the attracting case considered here, $\mathcal{M}_\eps$
retains the attracting property of $\mathcal{M}_0$. Trajectories approach
$\mathcal{M}_\eps$ rapidly in the fast direction and then evolve along it on the
slow time scale.}
\label{fig:slow_manifold_geometry}
\end{figure}

Next, we examine the conditions of Theorem~\ref{thm1}. The layer problem~\eqref{e6} admits the solution,
\begin{equation}
\vc w(\tau) = 
\begin{cases}
\vc w(0)e^{-\tau},& \alpha=1 \quad \text{(linear drag)},\\
\frac{\vc w(0)}{1+\|\vc w(0)\|\tau},& \alpha=2\quad  \text{(quadratic drag)}.
\end{cases}
\label{e7}
\end{equation}

In the linear drag case ($\alpha=1$), solutions decay exponentially fast towards the critical manifold $\mathcal{M}_0$. Thus, the critical manifold is normally hyperbolic and attracting, so that the assumptions of Theorem~\ref{thm1} hold. In the quadratic drag case ($\alpha=2$), 
the solutions still converge towards the invariant manifold $\mathcal M_0$ as $\tau\to\infty$, but the rate of convergence is only algebraic. The critical manifold is therefore not normally hyperbolic, rendering Theorem~\ref{thm1} inconclusive. As such, it is not immediately clear whether the invariant manifold persists for $\epsilon>0$.

In Section~\ref{sec:main_result}, we use the blowup method to recover normal hyperbolicity for the quadratic Maxey--Riley equation. This geometric desingularization technique has its roots in algebraic geometry (see, e.g., Chapter 1.4 of \cite {Hartshorne1977}). It has subsequently been used to resolve singularities of vector fields \cite{dumortier1977,takens1974} and to study slow-fast systems near non-hyperbolic sets \cite{dumortier1996,jardon2019,krupa2001,kuehn2015}. 
Its application in the context of inertial particle dynamics is carried out here for the first time.

\section{Slow manifold construction for quadratic drag}
\label{sec:main_result}
Henceforth, we focus on the quadratic drag case ($\alpha=2$). In Section~\ref{sec:prelims}, we showed that the critical manifold $\mathcal M_0$ is not normally hyperbolic in the quadratic drag formulation. We made this assessment by examining the exact solution of the layer problem~\eqref{e6}. Alternatively, one can arrive at the same conclusion by linearizing the dynamics $\vc w'=-\|\vc w\|\vc w$ normal to the critical manifold $\mathcal M_0$. This linearization yields 
\begin{equation}
	\mathrm D_{\vc w}\left[-\|\vc w\|\vc w\right]_{\vc w=\vc 0}=\vc 0_{3\times 3},
	\label{eq:quad_degenerate_linearization}
\end{equation}
where $\mathrm D_{\vc w}$ denotes the Jacobian with respect to $\vc w$. Thus, $\mathcal M_0$ is a manifold of non-hyperbolic equilibria with a vanishing linearization in the normal direction.

Here, we use the blowup method to desingularize the vector field near the critical manifold through a very specific nonlinear change of coordinates. We divide our approach into three parts. First, we introduce the general form of the blowup transformation applicable to inertial particle dynamics with quadratic drag (Section~\ref{sec:3.1}). Then, we examine the physically relevant part of the blowup coordinates by introducing a rescaling chart (Section~\ref{sec:rescaling}). Next, we desingularize the transformed vector field in the rescaling chart (Section~\ref{sec:3.3}). Finally, we state and prove our main result in Theorem~\ref{thm2} (Section~\ref{sec:main_thm}).

\subsection{Blowup transformation}
\label{sec:3.1}

To set up the blowup construction, we first append the small parameter $\eps$ as a state variable in system \eqref{e5} with $\alpha=2$. This results in the system,
\begin{subequations}
\begin{align}
t'&=\eps,\\
{\vc x}' &= \eps(\vc u+\vc w),\\
{\vc w}' &=- \|\vc w\|\vc w-\eps\dr\left(\mdiff{\vc u}+\vc e_\g\right)-\eps(\vc w\cdot\grad)\vc u,\label{e9b}\\
\eps'&=0.
\end{align}
\label{e9}
\end{subequations}

Since the loss of normal hyperbolicity occurs in the fast variable $\vc w\to\vc 0$ as $\eps$ tends to zero, we blow up only $(\vc w,\eps)$, leaving the variables $(t,\vc x)$ unchanged. This motivates the following cylindrical blowup transformation.

\begin{definition}[Weighted cylindrical blowup]
Let $\vc F:\Omega_t\times\Omega_x\times\R^3\times\R\rightarrow\R^8$ be the vector field corresponding to the right-hand side of \eqref{e9}, and let 
\begin{equation}
	\mathbb{S}^3=\left\{(\bar{\mathbf w},\bar\epsilon)\in \R^3\times\R:\|\bar{\mathbf w}\|^2+\bar\epsilon^2=1\right\}
\end{equation}
denote the unit $3$-sphere in $\R^4$. Consider the generalized cylindrical coordinate transformation,
\begin{equation}\label{e09}
\begin{aligned}
\bar{\phi}:\ & \Omega_t\times\Omega_x\times\mathbb{S}^3\times \mathbb R_0^+\rightarrow\Omega_t\times\Omega_x\times\R^3\times\R\\
&(t,\vc x,\bar{\vc w},\bar\eps,r)\mapsto (t,\vc x,\vc w,\eps):=(t,\vc x,r^a\bar{\vc w},r^b\bar \eps),
\end{aligned}
\end{equation}
with $a,b\in\N$. For $r>0$, the weighted cylindrical blowup, $\bar{\vc F}$ of the vector field $\vc F$, is given by
\begin{equation}
	\bar {\vc F} (t,\vc x,\bar{\mathbf w},\bar\epsilon,r)= [\mathrm D\bar{\phi}(t,\vc x,\bar{\mathbf w},\bar\epsilon,r)]^{-1}\vc F(\bar{\phi}(t,\vc x,\bar{\mathbf w},\bar\epsilon,r)),
	\label{blowupF}
\end{equation}
where $\mathrm D$ denote the Jacobian of the map.
\end{definition}

The term \emph{weighted} indicates that the radial exponents $(a,b)$ assigned to $\bar{\vc w}$ and $\bar\eps$ need not be equal \cite{kuehn2015}. The term cylindrical refers to the fact that $t$ and $\vc x$ pass through as Cartesian coordinates. Note that, for $r>0$, the map $\bar\phi$ is a diffeomorphism and $\bar{\vc F}$ is the pullback of the original vector field $\vc F$ under the map $\bar \phi$. In other words, the vector fields $\bar{\vc F}$ and $\vc F$ are smoothly conjugate for $r>0$.

Although the vector field $\bar{\vc F}$ extends continuously to $r=0$~\cite{takens}, this case merits special attention. The pre-image of the degenerate set $\mathcal{M}_0\times\{\eps=0\}$ under $\bar\phi$ is the cylinder,
\begin{equation}
\mathcal{C}_0=
\Omega_t\times \Omega_x\times\mathbb{S}^3\times \{r=0\}.
\end{equation}
Thus, each point in the degenerate set is replaced by a sphere of directions, and the union of these spheres forms the cylinder $\mathcal{C}_0$. The map $\bar\phi$ is the corresponding blow-down map, which collapses this cylinder back to the degenerate set in the original variables \cite{jardon2019,kuehn2015}. The upper left diagram in Figure~\ref{fig_blowup} depicts the blown up cylinder $\mathcal C_0$, while the upper right diagram shows its image under $\bar{\phi}$.

\begin{figure}[!t]
\centering
\includegraphics[width=\textwidth]{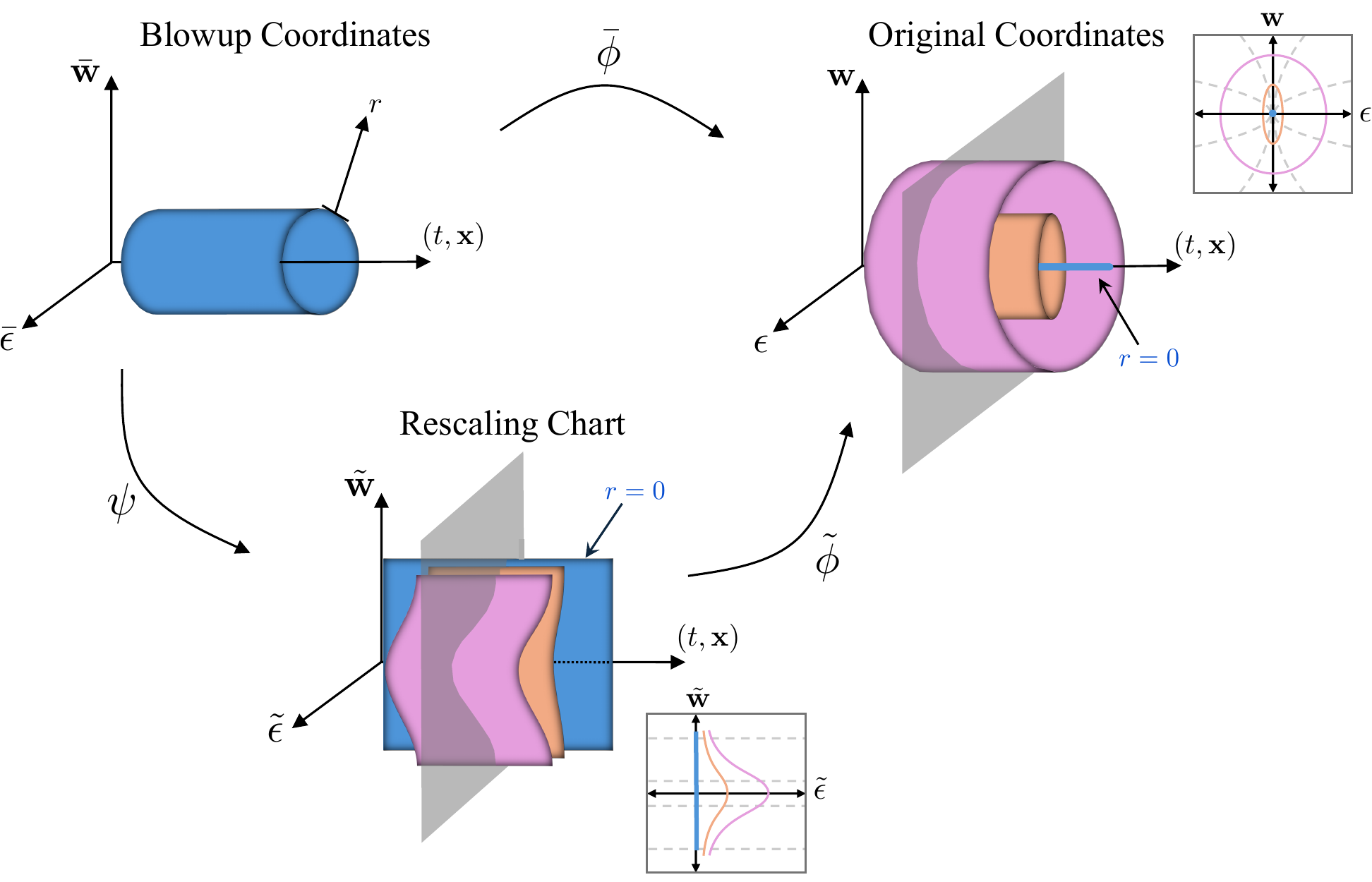}
\caption{
Geometry of the weighted blowup and the rescaling chart. 
The upper left diagram shows the blown up coordinates $(\bar{\mathbf w},\bar\epsilon,r)$, with $(t,\mathbf{x})$ passing through unchanged. 
The upper right diagram shows the image of these coordinates under the blow-down map $\bar\phi$ in the original variables $(\mathbf w,\epsilon)$. 
The lower middle diagram shows the image of the blown up coordinates under the rescaling chart $\psi$, with chart coordinates $(\tilde{\mathbf w},\tilde\epsilon)$. 
In both insets, colored curves are curves of constant $r$ and use the same colors as the constant $r$ shells in the corresponding three dimensional diagram. 
The dashed gray lines in the insets are curves of constant direction on the blown up sphere. 
For the exponents $a=1$ and $b=2$, the blow-down map $\bar\phi$ sends constant $r$ sections to ellipsoidal shells in the $(\mathbf w,\epsilon)$ variables and sends curves of constant direction to parabolic curves. 
The rescaling chart unwraps the $\bar\epsilon>0$ part of the blown up space into a Cartesian coordinate plane. 
In these coordinates, curves of constant direction become parallel dashed lines, while curves of constant $r$ extend across all values of the chart variable $\vcw$.
The chart blow-down map $\tilde\phi$ recovers the original variables by the mapping $\mathbf w=\tilde\epsilon\tilde{\mathbf w}$ and $\epsilon=\tilde\epsilon^2$. }
\label{fig_blowup}
\end{figure}

As we show in Section~\ref{sec:3.3}, if the exponents $(a,b)$ in \eqref{e09} are chosen judiciously, every component of the blowup vector field $\bar{\vc F}$ will contain a common maximal factor $r^k$ for some $k\in\mathbb N$. In other words, the $k$-jet of $\bar{\vc F}$ vanishes at $r=0$.
We remove this factor by defining the desingularized vector field
\begin{equation}
\bar {\vc G}(t,\vc x,\bar{\mathbf w},\bar\epsilon,r)=\frac{1}{r^k}\bar{\vc F}(t,\vc x,\bar{\mathbf w},\bar\epsilon,r),
\label{desing_vf}
\end{equation}
which is equivalent (but not necessary conjugate) to the vector field $\bar{\vc F}$ for $r>0$: the two vector fields have the same orbits up to a rescaling of time \cite[Proposition~3.1 and Remark~3.2]{takens1974}.
The desingularized vector field $\bar {\vc G}$  reveals dynamics that are hidden in $\bar {\vc F}$ by the vanishing radial factor $r^k$ as $r\to 0$. Removing this factor targets the degeneracy associated with the nilpotent normal linearization in \eqref{eq:quad_degenerate_linearization}. 
Since the vector fields $\vc F$, $\bar{\vc F}$, and $\bar{\vc G}$ describe the same trajectories up to a change of coordinates and a rescaling of time, the GSPT results obtained for $\bar{\vc G}$ translate back to the original vector field $\vc F$ \cite{kuehn2015}. As we show in Section~\ref{sec:3.3}, the appropriate choice of the blowup exponents are $a=1$ and $b=2$, resulting in the desingularization exponent $k=1$. In order to arrive at this conclusion, we first need to introduce the so-called \emph{rescaling chart} on the physically relevant component of the blowup cylinder $\mathcal C_0$.

\subsection{Rescaling chart}\label{sec:rescaling}
Note that, in the blow-down map $\bar\phi$, the variables $(\bar{\vc w},\bar\eps)$ are not independent coordinates because they lie on $\mathbb S^3$. To compute the pullback vector field $\bar{\vc F}$ explicitly, it is therefore more convenient to pass to a local chart on the sphere of directions $\mathbb S^3$. Since the physically relevant range for the small parameter is $\eps=\mu^{-1}>0$, we use the patch $\bar\eps>0$.

On the patch $\bar\eps>0$, we can rewrite the variable $\vc w=r^a\bar{\vc w}$ as $\vc w = (r\bar\eps^{1/b})^a (\bar{\vc w}/\bar\eps^{a/b})$.
Using the fact that $\eps=r^b\bar\eps$ is constant along trajectories, equation \eqref{e9b} can be written equivalently as
\begin{equation}
\begin{aligned}
\vc w'=(r\bar\eps^{1/b})^{a}\left(\frac{\bar{\vc w}}{\bar\eps^{a/b}}\right)' &=
- (r\bar\eps^{1/b})^{2a}
\left\|\frac{\bar{\vc w}}{\bar\eps^{a/b}}\right\|
\frac{\bar{\vc w}}{\bar\eps^{a/b}}
-(r\bar\eps^{1/b})^b
\dr\left(\mdiff{\vc u}+\vc e_\g\right) \\
&\qquad
-(r\bar\eps^{1/b})^{a+b}
\left(\frac{\bar{\vc w}}{\bar\eps^{a/b}}\cdot\grad\right)\vc u .
\end{aligned}
\end{equation}
The repeated appearance of the terms $r\bar\eps^{\,1/b}$ and $\bar{\vc w}/\bar\eps^{\,a/b}$ motivates us to define the local coordinates,
\begin{equation}
\epst := r\bar\eps^{1/b},
\qquad
\vcw:=\frac{\bar{\vc w}}{\bar\eps^{a/b}}.
\end{equation}
These coordinates define the directional chart on the patch $\bar\eps>0$, which is also referred to as the \emph{rescaling chart} in the blowup literature \cite{jardon2019,kuehn2015}; see the lower middle panel in Figure~\ref{fig_blowup}.

More precisely, let
\begin{equation}
\mathbb{S}^3_+=\{(\bar{\mathbf w},\bar\epsilon)\in \mathbb{S}^3:\bar\epsilon>0\},
\end{equation}
and define the rescaling chart by
\begin{equation}
\begin{aligned}
\psi:\ & \Omega_t\times\Omega_x\times\mathbb{S}^3_+\times \R_0^+
\rightarrow\Omega_t\times\Omega_x\times\R^3\times\R_0^+\\
&(t,\vc x,\bar{\mathbf w},\bar\epsilon,r)\mapsto(t,\vc x,\vcw,\epst):=\left(t,\vc x,\frac{\bar{\vc w}}{\bar\eps^{a/b}},r\bar\epsilon^{1/b}\right).
\end{aligned}
\end{equation}
The blow-down map from the rescaling chart to the original coordinate system is given by
\begin{equation}
\begin{aligned}
\tilde\phi:\ & \Omega_t\times\Omega_x\times\R^3\times\R_0^+
\rightarrow \Omega_t\times\Omega_x\times\R^3\times\R_0^+\\
&(t,\vc x,\vcw,\epst)\mapsto(t,\vc x,\vc w,\eps):=(t,\vc x,\epst^a\vcw,\epst^b).
\end{aligned}
\label{dirblowdown}
\end{equation}

As depicted in Figure~\ref{fig_blowup}, the maps $\bar \phi$, $\psi$, and $\tilde\phi$ form a commutative diagram, 
\begin{equation*}
\begin{tikzcd}	
	\begin{aligned}
		\Omega_t\times\Omega_x\times \mathbb S^3_+ \times \mathbb R_0^+ \\
			(t,\vc x, \bar{\vc w},\bar\epsilon,r)\qquad \arrow{rr}{\bar \phi}
		\end{aligned}
 \arrow[swap]{dr}{\psi}& & 
 \begin{aligned}
 \Omega_t\times\Omega_x\times \mathbb R^3 \times \mathbb R_0^+ \\
 (t,\vc x,\vc w,\epsilon)  \qquad
 \end{aligned}
 \\
	& 
	\begin{aligned}
	\Omega_t\times\Omega_x\times\R^3\times\R_0^+\\
	(t,\vc x,\tilde{\vc w},\tilde\epsilon)\arrow{ur}[swap]{\tilde \phi} \qquad
	\end{aligned}
	& 
\end{tikzcd} 
\end{equation*}
so that the map $\tilde\phi$ is the expression of the original blow-down map $\bar \phi$ in the rescaling chart and satisfies
\begin{equation}
\tilde\phi\circ\psi=\bar\phi\big|_{\bar\epsilon>0}.
\end{equation}

The chart coordinates $(\vcw,\epst)$ are related to the blowup and original variables by
\begin{equation}
\vcw=\frac{\bar{\vc w}}{\bar\epsilon^{a/b}}
=
\frac{\vc w}{\epsilon^{a/b}},\qquad\epst=r\bar\epsilon^{1/b}=\epsilon^{1/b}.
\end{equation}
Thus the rescaling chart blow-down takes the simple form $\vc w=\epst^a\vcw$ and $\epsilon=\epst^b$.
The lower middle diagram in Figure~\ref{fig_blowup} shows the image of the blown up coordinates under the rescaling chart $\psi$. The corresponding inset shows how curves of constant $r$ and constant weighted direction appear in this chart.

The pullback of $\vc F$ by the chart blow-down map $\tilde\phi$ is
\begin{equation}
\tilde {\vc F}(t,\vc x,\vcw,\epst) =
[\mathrm D\tilde\phi(t,\vc x,\vcw,\epst)]^{-1}
\vc F(\tilde\phi(t,\vc x,\vcw,\epst)).
\label{dirblowupF}
\end{equation}
We refer to $\tilde {\vc F}$ as the directional blowup of $\vc F$ in the rescaling chart. In components, this gives
\begin{subequations}
\begin{align}
t'&=\epst^b,\\
{\vc x}' &= \epst^b(\vc u+\epst^a\vcw),\\
\epst^a\vcw' &=-\epst^{2a}\|\vcw\|\vcw-\epst^{b}\dr\left(\mdiff{\vc u}+\vc e_\g\right)-\epst^{a+b}(\vcw\cdot\grad)\vc u,\label{weqn}\\
b\epst^{b-1}\epst'&=0.
\end{align}
\end{subequations}

The $\vcw$ equation \eqref{weqn} displays the powers of $\epst$ that determine the appropriate choice of exponents $(a,b)$ in order to achieve a desingularized vector field, as explained in Section \ref{sec:3.3}. 

\subsection{Desingularization}\label{sec:3.3}
We choose the exponents $a,b\in\mathbb N$ by comparing the powers of $\epst$ in \eqref{weqn}. The goal is to identify the lowest-order terms in $\epst$, since these are the terms that survive at $\epst=0$ after desingularization. Retaining more than one lowest-order term is important here because it allows the critical manifold at $\epst=0$ to occur at nonzero values of the rescaled fast variable $\vcw$, rather than collapsing to $\vcw=\vc 0$, where the original normal linearization vanished. 

The three exponents on the right-hand side of \eqref{weqn} are $2a$, $b$, and $a+b$. Since $a,b>0$, the power $a+b$ is larger than $b$ and therefore cannot be minimal. The only possible leading-order balance is therefore between the terms containing $\epst^{2a}$ and $\epst^b$. We balance these terms by choosing the smallest positive integer exponents satisfying $2a=b$, which yields 
 $a=1$ and $b=2$.
With these exponents, the rescaling chart vector field $\tilde {\vc F}$ from \eqref{dirblowupF} becomes
\begin{subequations}
\begin{align}
t'&=\epst^2,\\
{\vc x}' &= \epst^2(\vc u+\epst\vcw),\\
\vcw' &=-\epst\|\vcw\|\vcw-\epst\dr\left(\mdiff{\vc u}+\vc e_\g\right)-\epst^2(\vcw\cdot\grad)\vc u,\\
\epst'&=0.
\end{align}
\label{wtildeqns}
\end{subequations}

Note that every term has a common factor of $\tilde\eps$. For $\epst>0$, we define the desingularized rescaling chart vector field,
\begin{equation}
\tilde{\vc G}(t,\vc x,\vcw,\epst)=\frac{1}{\epst}\tilde{\vc F}(t,\vc x,\vcw,\epst).
\end{equation}
Recall from Section~\ref{sec:3.1} that his vector field can be smoothly extended to $\tilde\eps=0$. More precisely, 
$\tilde{\vc G}$ is the rescaling chart representation of the desingularized vector field $\bar{\vc G}$ defined in \eqref{desing_vf} with $k=1$. Equivalently, the division by $\epst$ corresponds to a time reparametrization. In particular, define the rescaled fast time,
\begin{equation}
\sigma=\epst\tau=\frac{t-t_0}{\epst}.
\end{equation}
Then \eqref{wtildeqns} can be written equivalently as
\begin{subequations}\label{e10}
\begin{align}
\diff{t}{\sigma}&=\epst,\\
\diff{\vc x}{\sigma} &= \epst(\vc u+\epst\vcw),\\
\diff{\vcw}{\sigma} &=-\|\vcw\|\vcw-\dr\left(\mdiff{\vc u}+\vc e_\g\right)-\epst(\vcw\cdot\grad)\vc u,\label{e10d}\\
\diff{\epst}{\sigma}&=0,
\end{align}
\end{subequations}
whose right-hand side is the desingularized vector field $\tilde{\vc G}$.

For every fixed $\epst>0$, system \eqref{e10} is equivalent to system \eqref{e9} under the transformations $\vc w=\epst\vcw$, $\eps=\epst^2$, and $\sigma=\epst\tau$. We now apply GSPT to \eqref{e10}. Taking $\epst\to0$ in \eqref{e10} gives
\begin{subequations}
\begin{align}
\diff{t}{\sigma}&=0,\\
\diff{\vc x}{\sigma}&= \vc 0,\\
\diff{\vcw}{\sigma} &=-\|\vcw\|\vcw-\dr\left(\mdiff{\vc u}+\vc e_\g\right).
\label{e11b}
\end{align}
\label{e11}
\end{subequations}
As in equation~\eqref{e6}, the variables $(t,\vc x)$ are invariant in the layer problem \eqref{e11}. However, its fixed points no longer correspond to $\tilde{\vc w}=\vc 0$. Assuming that,
\begin{equation}
\dr\left(\mdiff{\vc u}+\vc e_\g\right)\neq \vc 0,
\end{equation}
the fixed points of \eqref{e11b} are given by
\begin{equation}
\vcw = \vcw_1(t,\vc x)
:=
-\frac{\dr\left(\mdiff{\vc u}+\vc e_\g\right)
}{\left\|
\dr\left(\mdiff{\vc u}+\vc e_\g\right)
\right\|^{1/2}}.
\label{e11c}
\end{equation}
The corresponding critical manifold in the rescaling chart comprises all such fixed points,
\begin{equation}
\tilde{\mathcal{M}}_0=\left\{(t,\vc x,\vcw):\quad\vcw=-\frac{\dr\left(\mdiff{\vc u}+\vc e_\g\right)}{\left\|\dr\left(\mdiff{\vc u}+\vc e_\g\right)\right\|^{1/2}},
\quad
(t,\vc x)\in \Omega
\right\}.
\label{e12}
\end{equation}

To determine whether $\tilde{\mathcal{M}}_0$ is normally hyperbolic, we examine the linearization of the fast equation \eqref{e11b} with respect to $\vcw$,
\begin{equation}
\mathrm D_{\vcw}\left[-\|\vcw\|\vcw-\dr\left(\mdiff{\vc u}+\vc e_\g\right)\right]
=
-\|\vcw\|\vc I-\frac{\vcw\vcw^T}{\|\vcw\|}.
\label{rescaling_linearization}
\end{equation}
On the critical manifold $\tilde{\mathcal{M}}_0$, where $\tilde{\vc w}= \tilde{\vc w}_1$, the Jacobian \eqref{rescaling_linearization} has an eigenvalue $-2\|\vcw_1\|$ in the direction of $\vcw_1$. It also has an eigenvalue $-\|\vcw_1\|$ with geometric multiplicity of two, corresponding to the two-dimensional eigenspace orthogonal to $\vcw_1$. Therefore, the critical manifold $\tilde{\mathcal{M}}_0$ is normally hyperbolic and attracting. Contrast this with the critical manifold \eqref{eq:crit_man} in the original coordinate system which was not normally hyperbolic.

\subsection{Slow manifold for quadratic drag}\label{sec:main_thm}
The blowup construction, derived in Sections~\ref{sec:3.1}-\ref{sec:3.3}, transformed the original slow-fast system into an equivalent dynamical system whose critical manifold is normally hyperbolic. As such, Fenichel's GSPT applies in the rescaling chart and allows us to express the inertial particle velocity $\vc v(t) = \dot{\vc x}(t)$ as a graph over the spatiotemporal variables $(t,\vc x)$. We summarize these results in the following theorem.

\begin{theorem}[Main Result]\label{thm2}
Assume that $\vc u\in C^{\ell+1}(\Omega;\mathbb R^3)$ and $\dr\left(\mdiff{\vc u}+\vc e_\g\right)\neq \vc 0$ for all $(t,\vc x)\in \Omega$. Let $\tilde{\mathcal{M}}_0$ denote the critical manifold defined in \eqref{e12}. Writing $\epst=\eps^{1/2}$, there exists $\epst_0>0$ such that, for all $\epst\in[0,\epst_0)$, the following are true. 
\begin{enumerate}
	\item There exists a locally invariant, $C^\ell$-smooth manifold $\tilde{\mathcal{M}}_{\epst}$, diffeomorphic to $\tilde{\mathcal{M}}_0$, which is a graph over the spatiotemporal coordinates $(t,\vc x)$. Furthermore, $\tilde{\mathcal{M}}_{\epst}$ is $\mathcal{O}(\epst)$ close to $\tilde{\mathcal{M}}_0$ in Hausdorff distance, and is normally hyperbolic and attracting.

\item The equations of motion on the slow manifold $\tilde{\mathcal{M}}_\epst$ are given by
\begin{equation}
\dot{\vc x}
=
\vc u(t,\vc x)
+\epst\vcw_1(t,\vc x)
+\epst^2\vcw_2(t,\vc x)
+\ldots
+\epst^{\ell}\vcw_{\ell}(t,\vc x)
+\mathcal{O}(\epst^{\ell+1}).
\label{slowman_expansion}
\end{equation}
The functions $\vcw_k(t,\vc x)$ are given by the following recursive formulas:
\begin{subequations}
\begin{align}
\vcw_1
&=
-\frac{\dr(\mdiff{\vc u}+\vc e_\g)}
{\|\dr(\mdiff{\vc u}+\vc e_\g)\|^{1/2}}, \label{LOT}\\
\vcw_k
&=
-\frac{1}{\|\vcw_1\|}
\biggr[
\vc I
-
\frac{\vcw_1\vcw_1^T}{2\|\vcw_1\|^2}
\biggr]
\biggr[
\mdiff{\vcw_{k-1}}
+
(\vcw_{k-1}\cdot\nabla)\vc u\label{HOT}\\
&\qquad
+
\sum\limits_{j=1}^{k-2}
\left(
\frac{\vcw_{k-j}^\top\vcw_{j+1}-r_{k-j}r_{j+1}}
{2\|\vcw_1\|}
\vcw_1
+
r_{k-j}\vcw_{j+1}
+
(\vcw_j\cdot\nabla)\vcw_{k-j-1}
\right)
\biggr],\nonumber
\end{align}
\end{subequations}
where $2\le k\le \ell$ and the scalar coefficients $r_k$ are defined by
\begin{equation}
r_k
=
\frac{\vcw_k^\top\vcw_1}{\|\vcw_1\|}
+
\frac{1}{2\|\vcw_1\|}
\sum\limits_{j=1}^{k-2}
\left(
\vcw_{k-j}^\top\vcw_{j+1}
-
r_{k-j}r_{j+1}
\right),
\qquad 2\le k\le \ell.
\end{equation}
\end{enumerate}
\end{theorem}

\begin{proof}
See Appendix~\ref{app:A}.
\end{proof}

To preserve continuity with the previous section, we state Theorem~\ref{thm2} in the rescaling chart $(\vcw,\epst)$. But note that the slow manifold $\tilde{\mathcal M}_\epst$ in the rescaling chart, implies a corresponding slow manifold $\mathcal M_\eps=\tilde\phi(\tilde{\mathcal M}_\epst)$ (with the same properties) in the original coordinates $(\vc w,\eps)$, owing to the smoothness of the blow-down map $\tilde\phi$.

\begin{remark}
A few remarks regarding the assumption $\dr\left(\mdiff{\vc u}+\vc e_\g\right)\neq \vc 0$ are in order.
\begin{enumerate}
	\item At first glance, it may appear that this assumption is required for $\tilde{\vc w}_1(t,\vc x)$ to be well-defined. But in fact, $\tilde{\vc w}_1$ vanishes as $\dr\left(\mdiff{\vc u}+\vc e_\g\right)$ tends to zero, and therefore is well-defined. Instead, this assumption is required to ensure the normal hyperbolicity of the critical manifold $\tilde{\mathcal M}_0$. As discussed in Section~\ref{sec:3.3}, when $\tilde{\vc w}_1$ vanishes, the eigenvalues of the Jacobian~\eqref{rescaling_linearization} become zero and the critical manifold is no longer normally hyperbolic. As such, the question of its persistence becomes inconclusive using Fenichel's GSPT. Therefore, assumption $\dr\left(\mdiff{\vc u}+\vc e_\g\right)\neq \vc 0$ is required to ensure normal hyperbolicity of the critical manifold.
	
	\item In practice, the effective acceleration $\mdiff{\vc u}+\vc e_\g$ is unlikely to vanish; recall that $\vc e_\g$ is the constant unit vector pointing in the opposite direction of gravity. However, the special case of neutrally buoyant particles, which correspond to $\dr=0$, violates the assumptions of Theorem~\ref{thm2} and therefore requires special treatment. 
	We leave this special case to future work.
\end{enumerate}
\end{remark}

\begin{remark}
We also comment on the regularity of the slow manifold $\tilde{\mathcal M}_\epst$. Note that the vector field on the right-hand side of \eqref{e10} is only $C^1$ smooth at $\vcw=\vc 0$. However, the slow manifold $\tilde{\mathcal M}_\epst$ is a small perturbation of the critical manifold \eqref{e12} on which we have $\vcw\neq\vc 0$. Therefore, in the neighborhood of the critical manifold $\tilde{\mathcal M}_0$, the vector field \eqref{e10} is $C^\infty$ in $\vcw$ and $C^\ell$ in $(t,\vc x)$ because $\vc u\in C^{\ell+1}$. The slow manifold $\tilde{\mathcal M}_\epst$ inherits its smoothness from this vector field and therefore is also $C^\ell$ smooth.
\end{remark}

Theorem~\ref{thm2}, and especially equation~\eqref{slowman_expansion}, give the reduced inertial particle dynamics with quadratic drag. Since $\epst=\epsilon^{1/2}$ in the rescaling chart, the reduced dynamics contain half integer powers of the original small parameter $\eps=\mu^{-1}$. This contrasts with the linear drag case, where the slow manifold admits an expansion in integer powers of $\epsilon$. Furthermore, the functions $\tilde{\vc w}_k$ in the expansion~\eqref{slowman_expansion} are significantly different in the case of linear drag; cf. \cite{IP_haller08}.
Thus, quadratic drag changes both the geometric mechanism underpinning the slow manifold as well as the asymptotic structure of the reduced equations of motion.

\section{Numerical results}\label{sec:numerics}
In this section, we demonstrate the validity of our slow manifold reduction using two examples. In Section~\ref{sec:term_qui}, we consider a quiescent flow for which the exact solution to the inertial particle equation~\eqref{e2} is known. We show that our slow manifold reduction correctly predicts the terminal velocity of the particles in this case. 

We then turn, in Section~\ref{sec:TDM}, to a velocity field that models turbulent dispersion. In Section~\ref{sec:traj_conv}, we use this flow to numerically demonstrate the accuracy of the slow manifold reduction and the convergence rate towards the slow manifold. 
Finally, in Section~\ref{sec:source_inv}, we demonstrate a practical application of our reduced-order equations.
In particular, we use the turbulent flow and the corresponding slow manifold reduction for source inversion by tracing particles backward in time from their observed final positions. We show that this source inversion is not possible using the full-order model~\eqref{e2} because of its instability in backward time.

In the following, the full-order model (FOM) refers to \eqref{e2} with quadratic drag ($\alpha=2$). The reduced-order model (ROM) is obtained by replacing the inertial particle velocity with its slow manifold approximation~\eqref{slowman_expansion} from Theorem~\ref{thm2}. More specifically, for a ROM of order $m$, we have
\begin{equation}
\dot{\vc x}^{(m)}(t)=\vc u(t,\vc x^{(m)}(t))+\sum_{k=1}^{m}\eps^{k/2}\tilde{\vc w}_k(t,\vc x^{(m)}(t)),\qquad m=0,1,2,
\label{ROM}
\end{equation}
where the sum is omitted for $m=0$. Thus the zeroth-order ROM corresponds to fluid trajectories $\vc x^{(0)}(t)$, whereas higher-order ROMs include successive inertial corrections.

\subsection{Terminal velocity in quiescent flow}
\label{sec:term_qui}

We first consider \eqref{e2} with quadratic drag in a quiescent flow, $\vc u\equiv \vc 0$, and $\dr\ne0$. We consider the case where 
$\vc w(0)=\vc 0$, i.e., the relative velocity of the inertial particle vanishes at the initial time $t_0=0$. In this case, the exact solution to the FOM is known.

Under these assumptions, the FOM \eqref{e2} reduces to
\begin{equation}
\dot{\vc w}=-\frac{1}{\eps}\|\vc w\|\vc w-\dr\vc e_\g,\quad \vc w(0)=\vc 0,
\label{quiescent_vector}
\end{equation}
where $\eps=1/\mu$. The exact solution of \eqref{quiescent_vector} is
\begin{equation}
\vc w(t)=-\eps^{1/2}\frac{\dr}{|\dr|^{1/2}}\tanh\!\left[\left(\frac{|\dr|}{\eps}\right)^{1/2}t\right]\vc e_\g.
\end{equation}
As a result, we find the terminal velocity,
\begin{equation}
\lim_{t\to\infty}\vc w(t)=-\eps^{1/2}\frac{\dr}{|\dr|^{1/2}}\vc e_\g=:\vc w_\infty.
\end{equation}

We now compare this exact solution with the slow manifold reduction from Theorem~\ref{thm2}. The slow manifold has the expansion,
$\vc w_\eps=\eps^{1/2}\tilde{\vc w}_1+\eps\tilde{\vc w}_2+\eps^{3/2}\tilde{\vc w}_3+\cdots$.
For a quiescent flow, we have
\begin{equation}
\tilde{\vc w}_1=-\frac{\dr}{|\dr|^{1/2}}\vc e_\g,\qquad\tilde{\vc w}_k=\vc 0,\qquad k\geq 2.
\end{equation}
Therefore, the slow manifold reduction $\vc w_\eps$ coincides with the exact terminal velocity,
\begin{equation}
\vc w_\eps=-\eps^{1/2}\frac{\dr}{|\dr|^{1/2}}\vc e_\g=\vc w_\infty.
\end{equation}
Note that, in the quiescent flow, the reduced velocity $\vc w_\eps$ happens to be independent of space and time $(t,\vc x)$.

Finally, we examine the rate of convergence to the slow manifold by observing that,
\begin{equation}
\|\vc w(t)-\vc w_\eps\|= \left(\eps |\dr|\right)^{\frac12} \frac{2}{1+\exp\left[2|\dr/\eps|^{\frac12} t \right]}.
\end{equation}
Therefore, the exact solution $\vc w(t)$ of the FOM converges exponentially fast towards the slow manifold $\vc w_\eps$, as expected.

\subsection{Turbulent Dispersion}\label{sec:TDM}
In this section, we assess the slow manifold reduction of inertial particles suspended in a three-dimensional, time-dependent fluid flow. The fluid velocity field is based on the turbulent dispersion model (TDM) introduced by Lacorata, Mazzino, and Rizza \cite{lacorata2008}. The TDM was originally developed as a subgrid model for turbulent dispersion in large eddy simulations. Here, we use it instead as an analytic velocity field to examine the validity of our slow manifold reduction.

The TDM flow defines a three-dimensional incompressible velocity field from two stream functions, $\Psi_{\mathrm I}(x_2,x_3,t)$ and $\Psi_{\mathrm{II}}(x_1,x_3,t)$, which are spatially periodic and temporally quasi-periodic. The velocity field is then defined as $\vc u = -\nabla \times \pmb\Psi$, where $\pmb\Psi = [\Psi_{\mathrm I},\Psi_{\mathrm{II}},0]^\top$. 
In dimensionless variables, the components of the velocity field are given by
\begin{equation}
\begin{aligned}
u_1(t,\vc x)&=U_0+A_1\sin[k_1x_1-\gamma_1\sin(\omega_1 t)]\cos[k_3x_3-\gamma_3\sin(\omega_3t)],\\
u_2(t,\vc x)&=-A_2\sin[k_2x_2-\gamma_2\sin(\omega_2 t)] \cos[k_3x_3-\gamma_3\sin(\omega_3t)],\\
u_3(t,\vc x)&=-A_1\frac{k_1}{k_3}\cos[k_1x_1-\gamma_1\sin(\omega_1 t)]\sin[k_3x_3-\gamma_3\sin(\omega_3t)]\\
&\quad +A_2\frac{k_2}{k_3} \cos[k_2x_2-\gamma_2\sin(\omega_2 t)]\sin[k_3x_3-\gamma_3\sin(\omega_3t)],
\end{aligned}
\end{equation}
where $\vc x = (x_1,x_2,x_3)$.
Here, we slightly modify the TDM velocity field by introducing a mean flow $U_0$ in the $x_1$ direction. In the following, we set $U_0=1$.
Since the field is spatially periodic, and does not impose boundary conditions, it should be interpreted as a model for a turbulent flow away from any walls, rather than a boundary layer model. Following the parameter choices in \cite{lacorata2008}, we set
\begin{equation}
(\omega_1,\omega_2,\omega_3)
=
\left(2\pi,2\pi\sqrt{2},\frac{2\pi^2\sqrt{2}}{3}\right),
\quad
\gamma_1=\gamma_2=\gamma_3=\frac{\pi}{2},
\quad (k_1,k_2,k_3)=(2\pi,2\pi,4\pi).
\end{equation}
For the amplitudes, we choose $A_1=1/4$ and $A_2=1/8$.

\subsubsection{Comparing full and reduced trajectories in TDM}\label{sec:traj_conv}
We first compare FOM trajectories with zeroth-, first-, and second-order ROM trajectories in the TDM flow over the time interval $t\in[0,80]$ using $\eps^{1/2}=5\times 10^{-3}$. Figures~\ref{fig:dsf_trajectory_comparison}(a,c) show the trajectories for a bubble with $\dr=-0.5$, and Figures~\ref{fig:dsf_trajectory_comparison}(b,d) show the trajectories for an aerosol with $\dr=0.5$. All trajectories are initialized from a randomly chosen initial condition at $\vc x(0)\approx(-0.3245,0.5906,7.5604)$. For the FOM, the particle velocity is set initially equal to the fluid velocity, $\vc v(0)=\vc u(0,\vc x(0))$. 
\begin{figure}[!t]
	\centering
	\includegraphics[width=\textwidth]{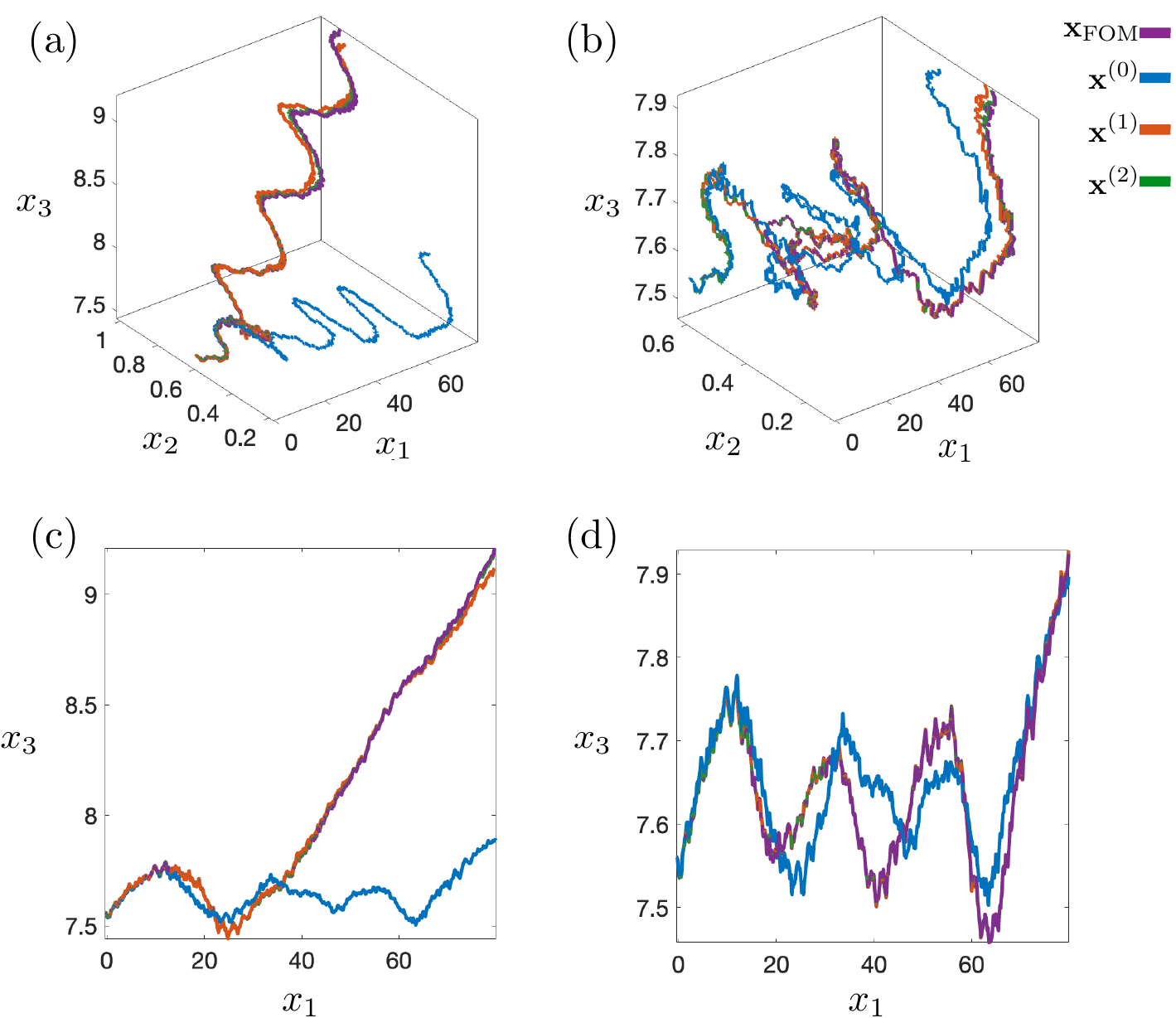}
	\caption{FOM and ROM trajectories in the TDM flow with $\eps^{1/2}=5\times 10^{-3}$ over the time interval $t\in[0,80]$. (a) Three dimensional trajectories for a bubble with $\dr=-0.5$. (b) Three dimensional trajectories for an aerosol with $\dr=0.5$. (c) Projection of the bubble trajectories onto the $(x_1,x_3)$ plane. (d) Projection of the aerosol trajectories onto the $(x_1,x_3)$ plane. For both bubbles and aerosols, the FOM initial velocity is set equal to the fluid velocity. The zeroth order ROM follows the fluid flow, while the first- and second-order ROMs include successive slow manifold corrections and more closely track the FOM trajectory.}
	\label{fig:dsf_trajectory_comparison}
\end{figure}

Recall that the zeroth-order ROM ($m=0$) coincides with the fluid trajectory and therefore deviates substantially from the FOM trajectory. On the other hand, the first- and second-order ROMs, which include successive inertial corrections, track the FOM trajectory more closely.
In fact, the trajectories corresponding to $m=1,2$ are so close to the truth $\vc x_{\mathrm{FOM}}$ that their differences are not readily visible in Figure~\ref{fig:dsf_trajectory_comparison}. In order to see the differences more clearly, Figure~\ref{fig:dsf_trajectory_error} shows the relative error, $\|\vc x^{(m)}(t)-\vc x_{\mathrm{FOM}}(t)\|/\|\vc x_{\mathrm{FOM}}(t)\|$, between the ROM trajectories and the truth.
The errors decrease consistently as the order of approximation $m$ increases.
\begin{figure}[!t]
\centering
\includegraphics[width=\textwidth]{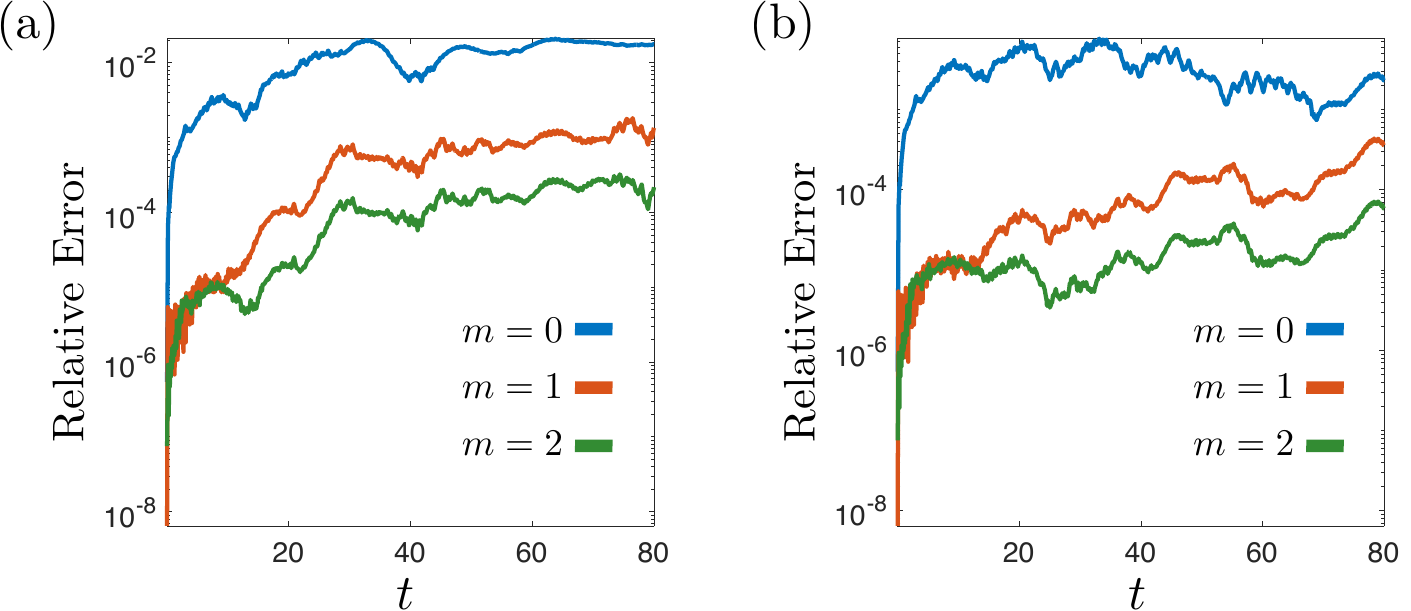}
\caption{The relative error, $\|\vc x^{(m)}(t)-\vc x_{\mathrm{FOM}}(t)\|/\|\vc x_{\mathrm{FOM}}(t)\|$, for the ROM trajectories shown in Figure~\ref{fig:dsf_trajectory_comparison}. (a) Error for a bubble with $\dr=-0.5$. (b) Error for an aerosol with $\dr=0.5$.}
\label{fig:dsf_trajectory_error}
\end{figure}

The above results compare the accuracy of the ROM for a single value of the small parameter $\eps$. Next, we assess the accuracy more systematically across different values of this parameter. The slow manifold expansion in Theorem~\ref{thm2} predicts that the ROM with order $m$ approximates the FOM velocity on the slow manifold with error,
\begin{equation}
\dot{\vc x}_{\mathrm{FOM}}-\dot{\vc x}^{(m)}
=
\mathcal{O}\!\left(\eps^{(m+1)/2}\right).
\end{equation}
Thus, for sufficiently small $\eps$, increasing the ROM order $m$ should improve the accuracy of the approximation to the FOM with the rate $\eps^{(m+1)/2}$. We test this theoretical prediction in the TDM flow with $\dr=-0.5$; the results corresponding to $\dr=0.5$ are similar and therefore are omitted here for brevity. For each $\eps>0$, we initialize the FOM at the same position $\vc x(0)$ as the ROMs and set its initial velocity $\vc v(0)$ equal to the corresponding ROM velocity. We then compare the FOM velocity with the ROM velocity over the time interval $t\in[0, 10]$. Figure~\ref{fig:convergence_study} shows the median over time of the velocity error, $\|\dot{\vc x}_{\mathrm{FOM}}-\dot{\vc x}^{(m)}\|$, and confirms the expected convergence rates, $\mathcal O(\eps^{(m+1)/2})$, for $m=0,1,2$.
\begin{figure}[!t]
\centering
\includegraphics[width=.4\textwidth]{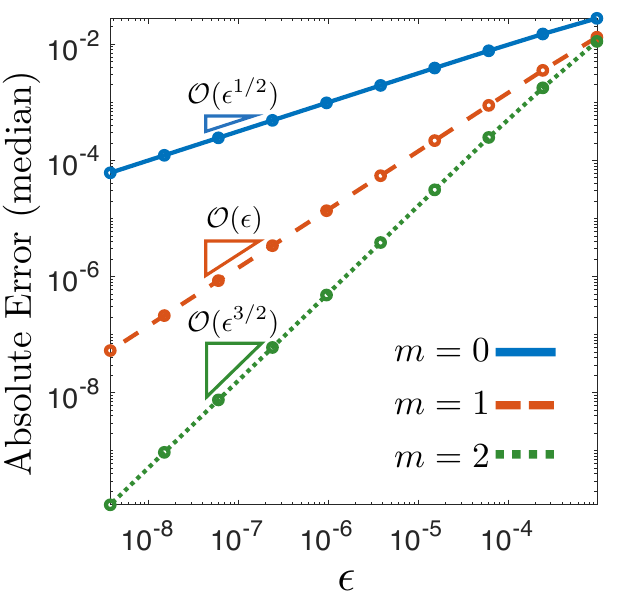}
\caption{Convergence of ROM velocities to the FOM velocity as $\eps\to0$ in the TDM flow with $\dr=-0.5$. For each ROM order $m=0,1,2$, the vertical axis shows the median of the absolute error, $\|\dot{\vc x}_{\mathrm{FOM}}(t)-\dot{\vc x}^{(m)}(t)\|$, over $t\in [0,10]$. The observed slopes agree with the slow manifold prediction: the errors scale like $\eps^{1/2}$, $\eps$, and $\eps^{3/2}$ for $m=0,1,2$, respectively.}
\label{fig:convergence_study} 
\end{figure}

Together, these numerical results demonstrate that the slow manifold reduction captures the inertial particle dynamics in a chaotic, three-dimensional, and time-dependent flow. The trajectory comparisons show quantitative agreement between the FOM and the slow manifold reduction, with higher-order corrections successively decreasing the relative error. The convergence study confirms the scaling of the error with $\eps$ as predicted by Theorem~\ref{thm2}.

\subsubsection{Source inversion using the reduced dynamics}
\label{sec:source_inv}
Source inversion concerns the inference of an inertial particle's release point from particle position observed at a later time. A primary motivation is source detection of environmental pollutants, with examples including aerosols in the atmosphere \cite{tang2009}, marine debris \cite{BeronVera2020}, olfactory tracking \cite{Reddy2022}, and urban contaminant plume source identification \cite{chow2008}. 

In the present setting, source inversion requires integrating particle trajectories backward in time from their observed final positions. Backward integration of the FOM~\eqref{e2} would require the particle relative velocity $\vc w(T)$ at the final time $T$, in addition to its position $\vc x(T)$. Although the final position can be easily observed, measuring the particle velocity is more challenging. Even if the final particle velocity is known, the FOM is unstable in backward time: in backward time, the FOM trajectories are repelled from the slow manifold. As a result, small errors in the observed relative velocity, which decay in forward time, are amplified in backward time, causing the FOM velocity to grow rapidly. On the other hand, since the ROM~\eqref{slowman_expansion} evolves on the slow manifold, it avoids this fast backward instability. In addition, the backward integration of this ROM only requires the final positions, and not the final velocity.

\begin{figure}[t!]
	\centering
	\includegraphics[width=\textwidth]{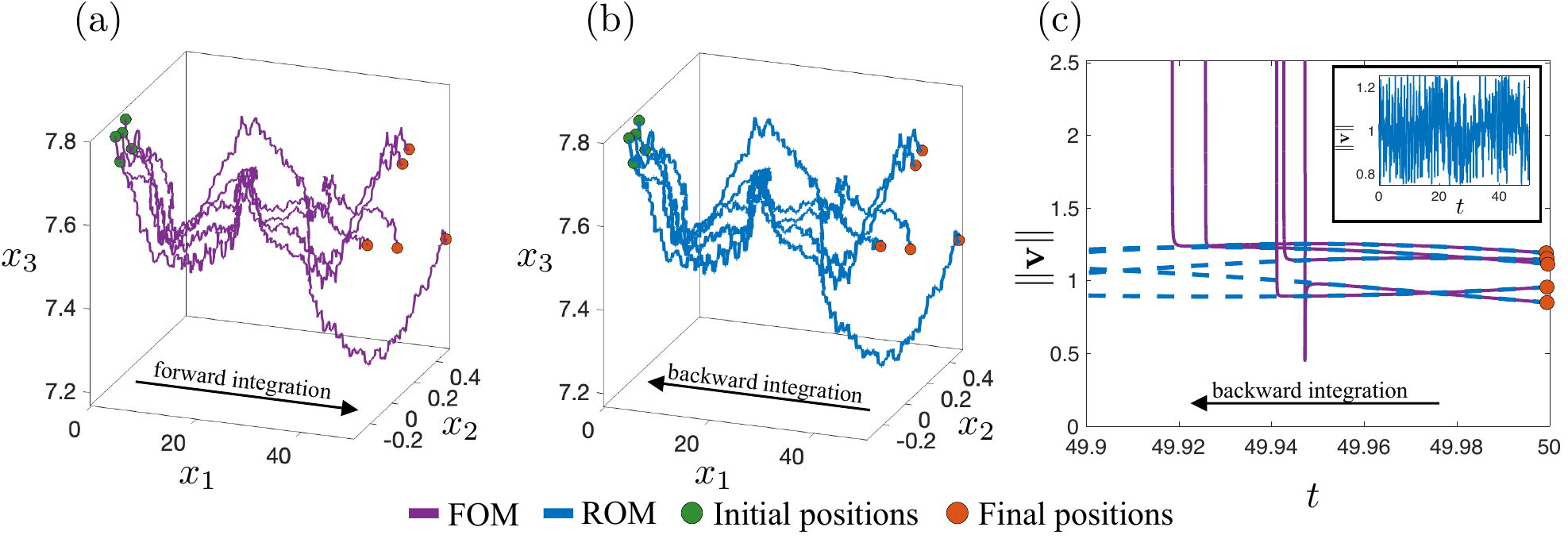}
	\caption{Source inversion in the TDM flow with $\dr=-0.5$ and $\eps^{1/2}=5\times 10^{-3}$ using the first-order ROM ($m=1$). Green dots mark the initial release locations at $t_0=0$, and orange dots mark the final observed locations at $T=50$. Purple curves correspond to the FOM, and blue curves correspond to the ROM, with dashed blue curves used in panel~(c) for contrast. (a) FOM trajectories integrated forward from the initial release locations to the final observed locations. (b) Backward time integration from the final observed locations; the ROM trajectories are visible, while the FOM trajectories cannot be visually resolved because of the rapid backward-time instability of the FOM. (c) Velocity magnitude during backward time integration near the final observation time. The FOM velocity rapidly grows in backward time, while the ROM velocity remains bounded over the full time interval as confirmed by the inset.}
	\label{fig:source_inv}
\end{figure}

We demonstrate this by considering the first-order ROM~\eqref{ROM}, corresponding to $m=1$, for source inversion in the TDM flow for a bubble with $\dr=-0.5$ and $\eps^{1/2}=5\times 10^{-3}$. We generate target data by initializing several FOM trajectories off the slow manifold at the initial time $t_0=0$ and integrating them forward in time to the final observation time $T=50$, as shown in Figure~\ref{fig:source_inv}(a). To infer the initial particle locations, the FOM is integrated backward from the final full state $(\vc x(T),\vc v(T))$, while the first-order ROM is integrated backward from the observed final positions $\vc x(T)$ using the reduced velocity field \eqref{ROM} with $m=1$. Figure~\ref{fig:source_inv}(b) shows the backward-time integration of the ROM, which accurately recovers the initial particle positions.

The backward-time trajectories of the FOM are omitted from Figure~\ref{fig:source_inv}(b) because they quickly diverge due to the aforementioned instability. This is demonstrated in Figure~\ref{fig:source_inv}(c) which shows this instability near the final observation time.
Within $0.1$ time units, the FOM backward trajectories rapidly develop large velocities, while the ROM trajectories remain bounded and trace backward toward the original release locations. The inset of Figure~\ref{fig:source_inv}(c) shows the particle velocity, according to the ROM, for one of the inertial particles over the entire backward integration time interval $[0,T]$. 

This example demonstrates the applicability of our slow manifold reduction to source inversion. The ROM remains bounded in backward time and accurately recovers the release locations, whereas the FOM rapidly develops instabilities in backward time.

\section{Conclusions}\label{sec:conclusion}
We derived a slow manifold reduction for inertial particle dynamics at high Reynolds numbers, where the drag force is quadratic in the relative velocity of the particle. For linear drag, the critical manifold is normally hyperbolic in the original variables, so that the geometric singular perturbation theory applies directly. In contrast, for quadratic drag, the normal hyperbolicity is lost, and GSPT is not immediately applicable. In order to resolve this issue, we applied a weighted cylindrical blowup, passed to the rescaling chart, and desingularized the transformed vector field. In the transformed system, the relevant critical manifold is normally hyperbolic and attracting, consequently proving that the critical manifold persists as an attracting invariant manifold, i.e., the slow manifold. The resulting reduced equation has a different asymptotic structure from its linear-drag counterpart. For example, in the quadratic drag case, the slow manifold expansion contains half-integer powers of the small parameter $\eps=1/\mu$, where $\mu$ is the drag coefficient.

We presented two examples demonstrating the validity of our slow manifold reduction. In the quiescent flow, the exact terminal velocity agrees with our slow manifold prediction, and the exact solution converges exponentially fast in time to this terminal velocity. In the time-dependent turbulent dispersion example, the first- and second-order reduced-order equations provide successively improved approximations of the full-order model as the small parameter $\eps$ decreases, with the numerical convergence rates matching the theory. 

Our work opens the door to application of reduced-order equations to inertial particle dynamics at high Reynolds numbers where the quadratic drag is the appropriate model, e.g., firebrand transport in wildfires~\cite{Koo2012FirebrandTransport,Mendez2022}. We showcased one such application here by considering the problem of source inversion: inferring initial positions of particles from their observed distribution at a later time. Although this problem is not tractable using the full-order model because of its backward-time instability, our reduced equations correctly identified the source of contamination.  Other applications, which should be pursued in the future, include the inference of particle clustering patterns from Eulerian quantities. This has been already accomplished in the linear drag case~\cite{IP_haller08,sapsis2010clustering}, but remains an open problem for quadratic drag.

The special case of neutrally buoyant particles, corresponding to the mass density ratio $\dr=0$, violates the assumptions of Theorem~\ref{thm2}. Future work should also address this special case. Our results indicate that neutrally buoyant particles may require their own specialized blowup construction.

\subsection*{Acknowledgments}
This work was partially supported through NSF awards DMS-2208541 and DMS-2220548.

\appendix
\section{Proof of Theorem~\ref{thm2}}\label{app:A}
The arguments presented in Sections \ref{sec:3.1}--\ref{sec:3.3} showed that the critical manifold $\tilde{\mathcal{M}}_0$ is normally hyperbolic and attracting. Therefore, Theorem~\ref{thm1} implies that, for sufficiently small $\epst$, this manifold persists as a locally invariant manifold $\tilde{\mathcal{M}}_\epst$ with the stated properties in the first part of Theorem~\ref{thm2}.

We now turn to the second part of the theorem and derive the expansion~\eqref{slowman_expansion}. Our arguments rely on the governing equation~\eqref{e10} in the rescaling chart, which we reproduce here for convenience: 

\begin{subequations}\label{app_desing_system}
	\begin{align}
		\diff{t}{\sigma}&=\epst,\label{app_t}\\
		\diff{\vc x}{\sigma} &= \epst(\vc u+\epst\vcw),\label{app_x}\\
		\diff{\vcw}{\sigma}&=-\|\vcw\|\vcw-\dr\left(\mdiff{\vc u}+\vc e_\g\right)-\epst(\vcw\cdot\grad)\vc u,\label{app_w}\\
		\diff{\epst}{\sigma}&=0\label{app_eps}.
	\end{align}
\end{subequations}

The slow manifold $\tilde{\mathcal{M}}_\epst$ is a graph over the variables $(t,\vc x)$ and therefore the relative velocity $\vcw$, restricted to this manifold, is given by a function $\vcw_\epst(t,\vc x)$. We consider the Taylor series expansion of this graph in the small parameter $\epst$,
\begin{equation}\label{wt_exp}
	\vcw_\epst(t,\vc x)
	=
	\vcw_1(t,\vc x)
	+
	\epst\vcw_2(t,\vc x)
	+
	\ldots
	+
	\epst^{\ell-1}\vcw_\ell(t,\vc x)
	+
	\mathcal{O}(\epst^\ell).
\end{equation}
Note that this expansion is legitimate since the right-hand side of~\eqref{app_desing_system} is smooth in $\epst$.

On the slow manifold $\tilde{\mathcal{M}}_\epst$, we have
\begin{equation}
	\dot{\vc x}\big|_{\tilde{\mathcal{M}}_\epst}
	=
	\frac{1}{\epst}
	\diff{\vc x}{\sigma}
	\biggr|_{\tilde{\mathcal{M}}_\epst}
	=
	\vc u(t,\vc x)+\epst\vcw_\epst(t,\vc x),
	\label{reduced_position_equation}
\end{equation}
where the first identity follows from the definition of the rescaled fast time~\eqref{app_t} and the second identity follows from~\eqref{app_x}.
Substituting the expansion~\eqref{wt_exp} in \eqref{reduced_position_equation} gives \eqref{slowman_expansion}.

It remains to derive the functional forms of $\vcw_k(t,\vc x)$ in~\eqref{wt_exp}. Substituting $\vcw_\epst$ in \eqref{app_w}, we obtain,
\begin{equation}\label{w_slowMan}
	\diff{}{\sigma}\vcw_\epst(t,\vc x(t))
	=
	-\|\vcw_\epst\|\vcw_\epst
	-\dr\left(\mdiff{\vc u}+\vc e_\g\right)
	-\epst(\vcw_\epst\cdot\grad)\vc u.
\end{equation}
In the remainder of this proof, we substitute the expansion \eqref{wt_exp} in the left- and right-hand side of \eqref{w_slowMan}, and equate similar powers of $\epst$ on both sides. 

\emph{Left-hand side of Eq. \eqref{w_slowMan}}:
Using the chain rule, we first observe that
\begin{equation}
\begin{aligned}
	\diff{\vcw_\epst}{\sigma}& =\partial_t\vcw_\epst\diff{t}{\sigma}+\left(\diff{\vc x}{\sigma}\cdot\nabla\right)\vcw_\epst\\
	& = \epst \left[ \partial_t \vcw_\epst + (\vc u\cdot \nabla)\vcw_\epst+\epst (\vcw_\epst\cdot \nabla)\vcw_\epst
	\right],
\end{aligned}
\end{equation}
where, for the last identity, we used equations \eqref{app_t} and \eqref{app_x}.
Upon substituting \eqref{wt_exp}, and recalling that $\mdiff{\ } = \partial_t+\vc u\cdot\nabla$, we obtain
\begin{equation}
	\diff{\vcw_\epst}{\sigma}=\sum_{n=1}^{\ell-1}\epst^n\left[\mdiff{\vcw_n}+\sum_{j=1}^{n-1}(\vcw_j\cdot\nabla)\vcw_{n-j}\right]+\mathcal{O}(\epst^\ell).
	\label{app_lhs_expansion}
\end{equation}

\emph{Right-hand side of Eq. \eqref{w_slowMan}}: First we write the norm as
\begin{equation}
	\|\vcw_\epst\|=\|\vcw_1\|+\epst r_2+\cdots+\epst^{\ell-1}r_\ell+\mathcal{O}(\epst^\ell),
	\label{app_norm_expansion}
\end{equation}
where
\begin{equation}
	r_k
	=
	\frac{\vcw_k^\top\vcw_1}{\|\vcw_1\|}
	+
	\frac{1}{2\|\vcw_1\|}
	\sum_{j=1}^{k-2}
	\left(
	\vcw_{k-j}^\top\vcw_{j+1}
	-
	r_{k-j}r_{j+1}
	\right), \quad 2\le k\le \ell.
	\label{app_rk}
\end{equation}
Therefore, the quadratic term can be expanded as
\begin{equation}
	\|\vcw_\epst\|\vcw_\epst=
	\|\vcw_1\|\vcw_1 
	+\sum_{n=1}^{\ell-1}\epst^n\left[\|\vcw_1\|\left(\vc I+\frac{\vcw_1\vcw_1^\top}{\|\vcw_1\|^2}\right)\vcw_{n+1}+\vc c_{n+1}\right]+\mathcal{O}(\epst^\ell),
\label{app_quad_expansion}
\end{equation}
where
\begin{equation}
\vc c_k
=
\sum_{j=1}^{k-2}
\left[
\frac{
	\vcw_{k-j}^\top\vcw_{j+1}
	-
	r_{k-j}r_{j+1}
}
{2\|\vcw_1\|}
\vcw_1
+
r_{k-j}\vcw_{j+1}
\right],\quad 2\le k\le \ell.
\label{app_ck}
\end{equation}

Finally, the right-hand side of equation~\eqref{w_slowMan} becomes
\begin{equation}
	\begin{aligned}
		-\|\vcw_1\|\vcw_1 & -\dr\left(\mdiff{\vc u}+\vc e_\g\right)\\
		&-\sum_{n=1}^{\ell-1}\epst^n\left[\|\vcw_1\|\left(\vc I+\frac{\vcw_1\vcw_1^\top}{\|\vcw_1\|^2}\right)\vcw_{n+1}+\vc c_{n+1} + \vcw_n\cdot \nabla \vc u\right]+\mathcal{O}(\epst^\ell).
	\end{aligned}
	\label{w_slowMan_rhs}
\end{equation}

\emph{Equating two sides of Eq. \eqref{w_slowMan}}: Equating equations \eqref{app_lhs_expansion} and \eqref{w_slowMan_rhs} implies, 
\begin{equation}
	\begin{aligned}
		&\|\vcw_1\|\vcw_1 +\dr\left(\mdiff{\vc u}+\vc e_\g\right) \\
		&+\sum_{n=1}^{\ell-1}\epst^n\Big[\|\vcw_1\|\left(\vc I+\frac{\vcw_1\vcw_1^\top}{\|\vcw_1\|^2}\right)\vcw_{n+1}+\vc c_{n+1} + \vcw_n\cdot \nabla \vc u +\mdiff{\vcw_n}+\sum_{j=1}^{n-1}(\vcw_j\cdot\nabla)\vcw_{n-j}\Big]\\
		&+\mathcal{O}(\epst^\ell)=\vc 0.
	\end{aligned}
	\label{w_slowMan_eq}
\end{equation}
For this identity to hold for all $\epst\in(0,\epst_0)$, each term multiplying a different power of $\epst$ must vanish. For $\epst^0$ term, we therefore obtain, 
\begin{equation}
\|\vcw_1\|\vcw_1+\dr\left(\mdiff{\vc u}+\vc e_\g\right)=\vc 0,
\end{equation} 
or equivalently,
\begin{equation}
	\vcw_1=-\frac{\dr\left(\mdiff{\vc u}+\vc e_\g\right)}{\left\|\dr\left(\mdiff{\vc u}+\vc e_\g\right)\right\|^{1/2}}.
	\label{app_w1}
\end{equation}
This coincides with Eq.~\eqref{LOT} in Theorem~\ref{thm2}.

For the terms multiplying $\epst^n$, with $n\geq 1$, we obtain
\begin{equation}
	\begin{aligned}
		\|\vcw_1\|\left(\vc I+\frac{\vcw_1\vcw_1^\top}{\|\vcw_1\|^2}\right)\vcw_{n+1}= 
		-\mdiff{\vcw_n}-(\vcw_n\cdot\nabla)\vc u-\sum_{j=1}^{n-1}(\vcw_j\cdot\nabla)\vcw_{n-j}-\vc c_{n+1}.
	\end{aligned}
	\label{app_match_n}
\end{equation}
Since $\vcw_1\ne \vc 0$, the matrix multiplying $\vcw_{n+1}$ is invertible, and
\begin{equation}
	\left[\|\vcw_1\|\left(\vc I+\frac{\vcw_1\vcw_1^\top}{\|\vcw_1\|^2}\right)\right]^{-1}=\frac{1}{\|\vcw_1\|}\left[\vc I-\frac{\vcw_1\vcw_1^\top}{2\|\vcw_1\|^2}\right].
	\label{app_inverse}
\end{equation}
Using this inverse, and setting $k=n+1$, we finally obtain
\begin{equation}
	\begin{aligned}
		\vcw_k
		&=-\frac{1}{\|\vcw_1\|}\left[\vc I-\frac{\vcw_1\vcw_1^\top}{2\|\vcw_1\|^2}\right]\biggr[\mdiff{\vcw_{k-1}}+(\vcw_{k-1}\cdot\nabla)\vc u\\
		&\qquad\qquad+\sum_{j=1}^{k-2}\biggr(\frac{\vcw_{k-j}^\top\vcw_{j+1}-r_{k-j}r_{j+1}}{2\|\vcw_1\|}\vcw_1+r_{k-j}\vcw_{j+1}+(\vcw_j\cdot\nabla)\vcw_{k-j-1}\biggr)\biggr],
	\end{aligned}
	\label{app_wk}
\end{equation}
for $2\le k\le \ell$. This coincides with the recursive formula \eqref{HOT} and completes the proof of Theorem~\ref{thm2}.


 \end{document}